\documentclass[12 pts]{article}

\usepackage{amsmath, amsthm, amssymb,amsfonts,latexsym,cancel}
\usepackage{color}
\usepackage{array}
\usepackage[bottom]{footmisc}
\usepackage[table]{xcolor}
\usepackage{ifsym}
\usepackage{float}
\usepackage[all]{xy}
\usepackage{ae,aecompl}
\usepackage[pdftex]{graphicx}
\DeclareGraphicsExtensions{.png,.pdf,.jpg,.bmp}
\usepackage{epsfig}
\usepackage{algorithm}
\usepackage{listings}
\usepackage{multirow}
\usepackage{rotating}
\usepackage{booktabs}
\usepackage{fancyhdr}
\usepackage[pdftex=true,colorlinks=false,plainpages=false]{hyperref}
\usepackage[margin=1in]{geometry}
\usepackage{textcomp}
\usepackage[T1]{fontenc}
\usepackage{relsize}
\usepackage{makecell}

\usepackage{mathtools}

\usepackage[justification=centering]{caption}
\usepackage{subcaption} 
 
\usepackage{titlesec}
\titlespacing{\paragraph}{%
  0pt}{
  0.0\baselineskip}{
  1em}

\usepackage{amsmath,amsthm,amssymb}
\usepackage{graphicx}
\usepackage{epsfig}
\usepackage{multirow}
\usepackage{caption}
\usepackage{subcaption}

\usepackage{authblk}
\newtheorem{theorem}{Theorem}
\newtheorem{corollary}{Corollary}
\newtheorem{lemma}{Lemma}
\newtheorem{definition}{Definition}
\newtheorem{proposition}{Proposition}

\newcommand{\Z}{\mathbb{Z}}

\usepackage[noend]{algpseudocode}

\makeatletter
\def\BState{\State\hskip-\ALG@thistlm}
\makeatother

\DeclareMathOperator*{\argmin}{Argmin}

\usepackage[titletoc,title]{appendix}

\begin{document}

\bibliographystyle{abbrv}

\title{\vspace{-2em}A Linear Programming Based Approach to the Steiner Tree Problem with a Fixed Number of Terminals}
\author[1]{Matias Siebert}
\author[1]{Shabbir Ahmed}
\author[1]{George Nemhauser}
\affil[1]{H. Milton Stewart School of Industrial and Systems Engineering, Georgia Institute of Technology, Atlanta, GA}
\date{}
\maketitle

\begin{abstract}
\noindent We present a set of integer programs (IPs) for the Steiner tree problem with the property that the best solution obtained by solving all, provides an optimal Steiner tree. Each IP is polynomial in the size of the underlying graph and our main result is that the linear programming (LP) relaxation of each IP is integral so that it can be solved as a linear program. However, the number of IPs grows exponentially with the number of terminals in the Steiner tree. As a consequence, we are able to solve the Steiner tree problem by solving a polynomial number of LPs, when the number of terminals is fixed.
\end{abstract}

\section{Introduction}\label{Introduction}

Given an undirected graph $G=(V,E)$, where $V$ is the set of nodes, $E$ is the set of edges, $c_e\geq 0$ is the cost of each edge $e\in E$, and $R\subseteq V$ is a set of {\it terminal nodes}, a Steiner tree is a subgraph of $G$, which is a tree such that all terminal nodes are connected. The Steiner tree problem is to find such a tree that minimizes the sum of the edge costs over all edges in the tree. The tree may contain nodes in $V\backslash R$, which are called {\it Steiner nodes}. The Steiner tree problem is NP-Hard \cite{karp1972reducibility}, in fact, it is even NP-Hard to find approximate solutions whose cost is within a factor $\frac{96}{95}$ of the cost of an optimal solution \cite{bern1989steiner,chlebik2008steiner}.

An important line of research has been to study integer linear programming formulations (IP) for this problem. In \cite{goemans1993catalog} the authors provide a catalog of Steiner tree formulations, and show the equivalence of some of these formulations. In \cite{goemans1994steiner} the author studies the vertex-weighted version of the undirected Steiner tree problem, and presents a complete description of the polytope when the graph is {\it series-parallel}. By projecting this formulation, some facet-defining inequalities for the Steiner tree polytope are obtained. In \cite{polzin2001comparison}, the authors compare different formulations of the Steiner tree problem in terms of the strength of their linear relaxations (LP), and propose a new polynomial size formulation with a stronger LP relaxation than the ones analyzed in their paper.

In every Steiner tree, we can pick an arbitrary node $r\in R$ and find a direction of the edges, such that we can construct a unique arborescence rooted at $r$ whose leaf nodes correspond to nodes in $R\backslash\{r\}$. Note that an $r$-arborescence is a directed tree where every node in the arborescence is reachable from the root node $r$. Consequently, the Steiner tree problem can be modeled as an $r$-arborescence whose leaves are nodes in $R\backslash\{r\}$ and where all terminal nodes are reached from $r$. A well studied formulation that uses this fact, is the {\it Bidirected Cut} formulation \cite{edmonds1967optimum,wong1984dual}. This is a compact and simple formulation, whose integrality gap is known to be at most 2, but is believed to be close to one. The best known lower bound on the integrality gap for this formulation is $\frac{36}{31}\approx 1.16$ \cite{byrka2013steiner}. An upper bound on the integrality gap for this formulation has been found for special cases. For {\it quasi-bipartite} graphs, the integrality gap is upper bounded by $\frac{3}{2}$ \cite{rajagopalan1999bidirected}, and for {\it claw-free} graphs it is upper bounded by $\ln(4)$ \cite{feldmann2016equivalence}.

Another well studied formulation, is the so called {\it Hypergraphic} formulation \cite{konemann2011partition,polzin2003steiner,warme1998spanning}. This formulation has one variable for each {\it component} of the graph, where a {\it component} is a tree whose leaves correspond to terminal nodes. The number of variables and constraints for this formulation are proportional to the number of {\it components}, which is exponential in the number of terminal nodes. On the positive side, we have that the {\it Hypergraphic} formulation is stronger than the {\it Bidirected Cut} formulation \cite{polzin2003steiner}. The integrality gap of the {\it Hypergraphic} formulation is lower bounded by $\frac{8}{7}\approx1.14$ \cite{konemann2011partition} and upper bounded by $\ln(4)$ \cite{goemans2012matroids}. Moreover, in \cite{feldmann2016equivalence} the authors show that in {\it claw-free} graphs both formulations are equivalent. On the down side, solving the {\it Hypergraphic} formulation is strongly NP-Hard \cite{goemans2012matroids}. To address this problem, researchers have proposed to solve it only considering components with at most $k$ leaves, where $k$ is a fixed value. This considerably reduces the number of constraints and variables, and the solution of this restricted formulation is proven to be a factor of $\rho_k$ of an optimum solution, where $\rho_k\leq 1+\frac{1}{\lfloor\log_2(k)\rfloor}$ \cite{borchers1997thek}. A complete review of the {\it Hypergraphic} formulation and its variants can be found in \cite{chakrabarty2010hypergraphic}.

Finally, in \cite{dreyfus1971steiner} the authors propose an algorithm that finds an optimal solution for the Steiner tree problem in $\mathcal{O}(n^3+n^22^{b}+n3^{b})$, where $n=|V|$, $m=|A|$, and $b+1=|R|$. The result of this paper suggests that, for a fixed number of terminals, there may exist a polynomial size LP formulation of the Steiner tree problem, which is the motivation for the results presented in this paper.

Our main contribution is a set of independent IPs with the property that the best solution obtained by solving all, provides an
optimal Steiner tree. Each IP is polynomial in the size of the underlying graph and our main result is that the LP relaxation of
each IP is integral. Furthermore, the set of LPs can be solved in parallel. The drawback, is that the number of LPs grows exponentially with the number of terminal nodes so the complete algorithm is polynomial only for a fixed number of
terminals.

The remainder of this paper is organized as follows. Section \ref{Sol_Structure} analyzes the structure of Steiner trees, and also introduces definitions used in the paper. Section \ref{Proposed_Model} presents the proposed model, the proof of integrality, and an analysis of the size of the problem. In Section \ref{Results} we present computational experiments obtained by solving the Steiner tree problem using our proposed formulation. Section \ref{Conclusions} gives conclusions.

\section{Solution Structure}\label{Sol_Structure}

As mentioned before, several formulations use the fact that every Steiner tree contains an $r$-arborescence, some of which use a flow-based formulation, and others a cut-based formulation. Both of those common approaches work with the directed graph version $D$ of $G$. In our approach, we exploit two main properties that every solution must have. 
\begin{enumerate}
\item We can always pick an arbitrary node $r\in R$ as our root node, and find the min-cost $r$-arborescence that spans $R\backslash\{r\}$. Moreover, this $r$-arborescence can be modeled as a Min-cost Multi-Commodity Network Flow, where we have one commodity for every node in $R\backslash\{r\}$, and we want to send 1 unit of flow from $r$ to the corresponding node $i\in R\backslash\{r\}$.
\item Since the solution must correspond to an $r$-arborescence, there must be exactly one path from the root node $r$ to each one of the nodes in $R\backslash\{r\}$, which means that if any subset $S$ of the corresponding commodities of nodes in $R\backslash\{r\}$ share a path from $r$ to node $i\in V$, and then split in node $i$, then they will never meet again.  For example, if $S=\{k_1,k_2,k_3\}$ and in node $i$ the set $S$ splits into $S_1=\{k_1,k_2\}$ and $S_2=\{k_3\}$, then $k_1$ and $k_2$ will never share an arc again with $k_3$.
\end{enumerate}

Using these two properties, our formulation is based on two type of decisions, $(i)$ when we split the subsets of commodities, and $(ii)$ into which subsets of commodities. Note that all the commodities have the same source node, so they all start together. At some point, the set of all commodities will split into different subsets of commodities, and those subsets will also split into other subsets, and so on, until we have each commodity by itself. Our idea is to include a variable that tells us where a subset splits, and the partition that it uses. For instance, say that at some point we have the set $S=\{k_1,k_2,k_3\}$ of commodities that share a path. Eventually, $S$ is split, so the model has to decide where the split is going to happen. In addition, the model has to decide the way that $S$ is split. Note that there are 4 possible partitions of $S$.
\begin{itemize}
\item[(i)] $\{k_1,k_2,k_3\}\rightarrow\left\{\{k_1,k_2\},\{k_3\}\right\}$
\item[(ii)] $\{k_1,k_2,k_3\}\rightarrow\left\{\{k_1,k_3\},\{k_2\}\right\}$
\item[(iii)] $\{k_1,k_2,k_3\}\rightarrow\left\{\{k_2,k_3\},\{k_1\}\right\}$
\item[(iv)] $\{k_1,k_2,k_3\}\rightarrow\left\{\{k_1\},\{k_2\},\{k_3\}\right\}$
\end{itemize}
Two important observations are, first, we only consider proper partitions (i.e. a set is not a partition of itself), and second, the next partition that we can use will depend on the way that we partitioned $S$. If we take, for instance, partition $(i)$ then we will have to eventually partition the set $\{k_1,k_2\}$ into $\{k_1\},\{k_2\}$, but if we take partition $(iv)$ then we do not have to perform any more partitions.

\subsection{Notation}\label{Sec:Def}

In this section, we define the notation used in the rest of the paper. In Section \ref{Def_Example}, we provide an example to clarify the concepts introduced in this section.

We define $D=(V,A)$ as a directed graph such that, for every edge $\{i,j\}$ in $E$ we have two arcs $(i,j)$ and $(j,i)$ in $A$, where $V$ and $E$ are the set of nodes and edges of the original undirected graph $G$, and where $c_a=c_e$ for all $a\in A$ such that both of its end nodes correspond to the end nodes of $e\in E$. Let $r\in R$ be the arbitrarily selected root node, then we define a set $K$ of commodities, with one commodity for each of the nodes in $R\backslash\{r\}$. Let $s_k$ and $t_k$ be the source and sink nodes of commodity $k$, where $s_k=r$ for all $k\in K$, and $t_k$ is the corresponding node in $R\backslash\{r\}$. Let $S$ be the set of all subsets of $K$ with at least one element, i.e., $|S|=2^{|K|}-1$, and let $P$ be the set of all proper partitions of the sets of $S$ that have cardinality of at least 2. 

The underlying $r$-arborescence of any Steiner tree uses only a subset of the sets in $S$ (the set of all subsets of $K$), which contains the set of all commodities, and all the singletons. This collection of sets of $S$ forms a laminar family. A family $\mathcal{C}$ of sets is called {\it laminar} if for every $A,B\in \mathcal{C}$ we have $A\subseteq B$ or $B\subseteq A$ or $A\cap B=\emptyset$. Let $\mathcal{L}_b$ be the set of all possible laminar families when $|K|=b$. For laminar family $l\in\mathcal{L}_b$, let $S(l)$ be the set of subsets of $K$ that appear in laminar family $l$, and let $P(l)$ be the set of partitions used to split the subsets in $S(l)$.

For any $s'\in S(l)$ such that $|s'|\leq |K|-1$, we say that $\hat{s}$ is its {\it parent} set, if $s'$ is one of the sets obtained when we partition set $\hat{s}$ according to laminar family $l$. Note that there is only one parent set for every $s'\in S(l),\ |s'|\leq |K|-1$. Similarly, for any set $\hat{s}\in S(l)$ such that $|\hat{s}|\geq 2$, we say that $s'$ is a {\it child} set of $\hat{s}$, if we obtain $s'$ when we partition $\hat{s}$. Note that for every set $\hat{s}\in S(l),\ |\hat{s}|\geq 2$ we have at least 2 child sets. Let $S_l(p)$ be the set of all child nodes of partition $p$ in laminar family $l$, and let $P_l(s)$ be the partition $p$ that splits $s$ in laminar family $l$, which is defined only for sets $s$, such that $|s|\geq 2$. Finally, we say that a Steiner tree $T$ in $G$ follows an $(r,l)$ {\it structure}, if the arborescence rooted in $r$ contained in $T$ belongs to laminar family $l\in\mathcal{L}_b$, where $b=|R\backslash\{r\}|$.

\subsection{Example}\label{Def_Example}

Suppose we have an instance with 4 terminal nodes. Let $r$ be the root node, and let $K=\{k_1,k_2,k_3\}$. Then,

\begin{align*}
S=\left\{\underbrace{\{k_1\}}_{s_1};\underbrace{\{k_2\}}_{s_2};\underbrace{\{k_3\}}_{s_3};\underbrace{\{k_1,k_2\}}_{s_4};\underbrace{\{k_1,k_3\}}_{s_5};\underbrace{\{k_2,k_3\}}_{s_6};\underbrace{\{k_1,k_2,k_3\}}_{s_7}\right\}
\end{align*}
and
\begin{align*}
P=\left\{\underbrace{\{(k_1,k_2),(k_3)\}}_{p_1};\underbrace{\{(k_1,k_3),(k_2)\}}_{p_2};\underbrace{\{(k_2,k_3),(k_1)\}}_{p_3};\underbrace{\{(k_1),(k_2),(k_3)\}}_{p_4};\right.\\ 
\left.\underbrace{\{(k_1),(k_2)\}}_{p_5};\underbrace{\{(k_1),(k_3)\}}_{p_6};\underbrace{\{(k_2),(k_3)\}}_{p_7}\right\}\\
=\left\{\underbrace{\{s_4,s_3\}}_{p_1};\underbrace{\{s_5,s_2\}}_{p_2};\underbrace{\{s_6,s_1\}}_{p_3};\underbrace{\{s_1,s_2,s_3\}}_{p_4};\underbrace{\{s_1,s_2\}}_{p_5};\underbrace{\{s_1,s_3\}}_{p_6};\underbrace{\{s_2,s_3\}}_{p_7}\right\}
\end{align*}

In this case $p_1$, $p_2$, $p_3$ and $p_4$ are the partitions of the set whose cardinality is 3, which is the set $\{k_1,k_2,k_3\}$. The partitions $p_5$, $p_6$ and $p_7$ are the partitions of sets of cardinality 2, which are 3 different sets $\{k_1,k_2\},\{k_1,k_3\}$ and $\{k_2,k_3\}$.

Figure (\ref{Fig:Splitting_Seq_k3}) shows the tree representation of all the possible laminar families for the case $K=\{k_1,k_2,k_3\}$.

\begin{figure}[H]
\begin{subfigure}{.5\textwidth}
  \centering
  \includegraphics[width=.6\linewidth]{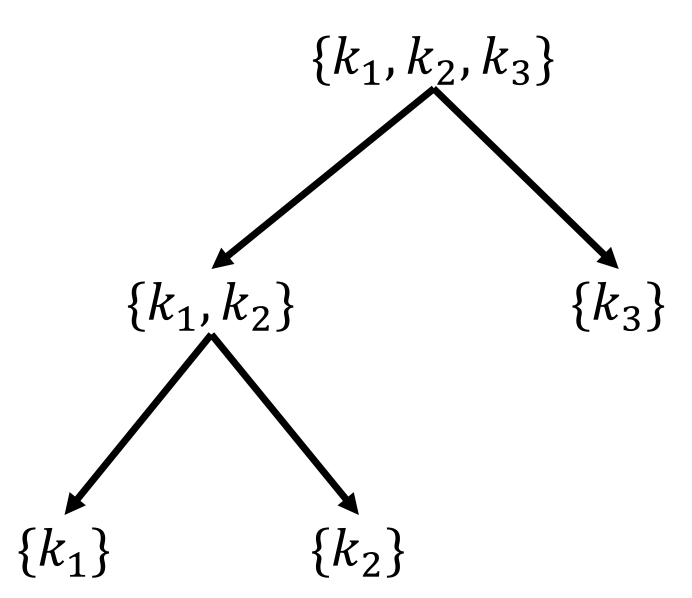}
  \caption{Tree representation of laminar family $l_1\in \mathcal{L}_3$.}
  \label{fig:sfig1}
\end{subfigure}
\begin{subfigure}{.5\textwidth}
  \centering
  \includegraphics[width=.6\linewidth]{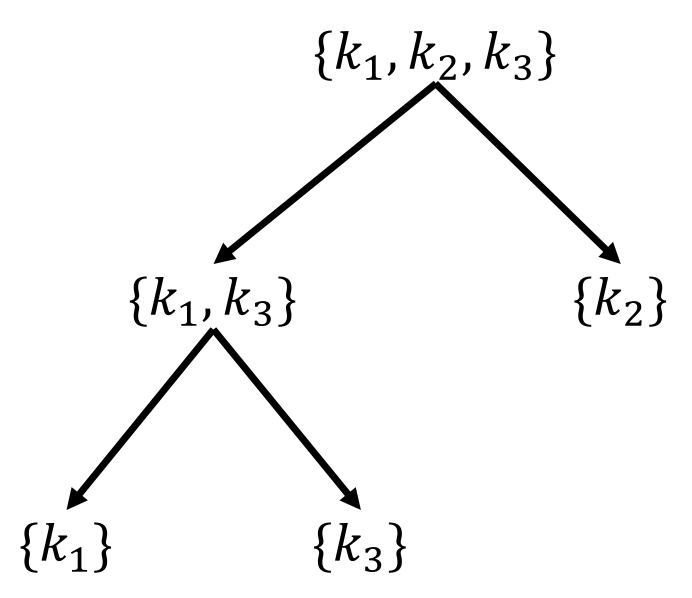}
  \caption{Tree representation of laminar family $l_2\in \mathcal{L}_3$.}
  \label{fig:sfig2}
\end{subfigure}\\
\begin{subfigure}{.5\textwidth}
  \centering
  \includegraphics[width=.6\linewidth]{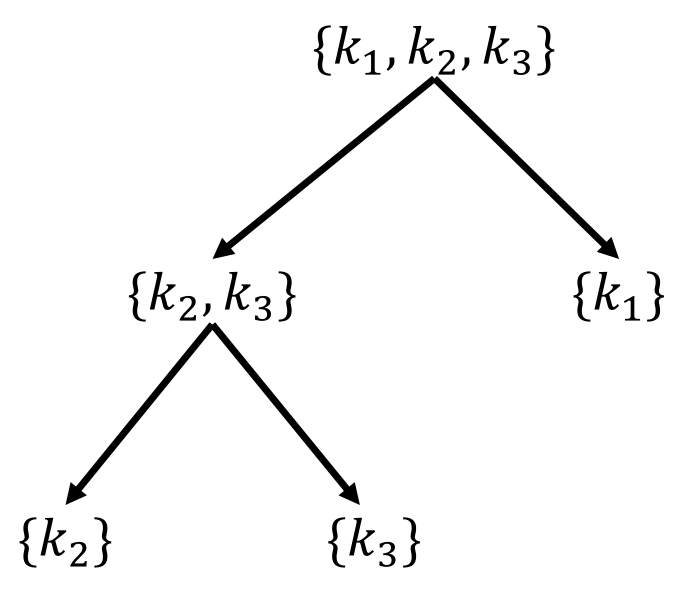}
  \caption{Tree representation of laminar family $l_3\in \mathcal{L}_3$.}
  \label{fig:sfig3}
\end{subfigure}
\begin{subfigure}{.5\textwidth}
  \centering
  \includegraphics[width=.6\linewidth]{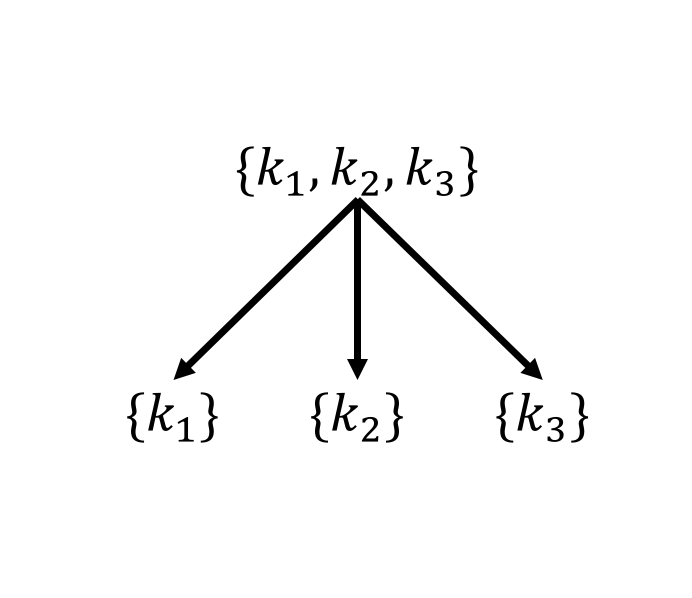}
  \caption{Tree representation of laminar family $l_4\in\mathcal{L}_3$.}
  \label{fig:sfig4}
\end{subfigure}
\caption{Tree representation of all possible laminar families for the case $K=\{k_1,k_2,k_3\}$}
\label{Fig:Splitting_Seq_k3}
\end{figure}
In this case, we have,
\begin{align*}
&P(l_1) = \{p_1,p_5\} && S(l_1) = \big\{\{k_1,k_2,k_3\};\{k_1,k_2\};\{k_1\};\{k_2\};\{k_3\}\big\} &= \{s_7,s_4,s_1,s_2,s_3\}\\
&P(l_2) = \{p_2,p_6\} && S(l_2) = \big\{\{k_1,k_2,k_3\};\{k_1,k_3\};\{k_1\};\{k_2\};\{k_3\}\big\} &= \{s_7,s_5,s_1,s_2,s_3\}\\
&P(l_3) = \{p_3,p_7\} && S(l_3) = \big\{\{k_1,k_2,k_3\};\{k_2,k_3\};\{k_1\};\{k_2\};\{k_3\}\big\} &= \{s_7,s_6,s_1,s_2,s_3\}\\
&P(l_4) = \{p_4\} && S(l_4) = \big\{\{k_1,k_2,k_3\};\{k_1\};\{k_2\};\{k_3\}\big\}&= \{s_7,s_1,s_2,s_3\}
\end{align*} 
Consider the laminar family $l_1\in\mathcal{L}_3$ shown in Figure (\ref{fig:sfig1}). We have that $\{k_1,k_2,k_3\}$ is the {\it parent} set of $\{k_1,k_2\}$ and $\{k_3\}$, and that $\{k_1,k_2\}$ is the {\it parent} set of $\{k_1\}$ and $\{k_2\}$. On the other hand, sets $\{k_1\}$ and $\{k_2\}$ are the {\it child} sets of $\{k_1,k_2\}$, and, set $\{k_1,k_2\}$ and $\{k_3\}$ are {\it child} sets of $\{k_1,k_2,k_3\}$.
Finally, we have
\begin{align*}
S_{l_1}(p_1) &= \big\{\{k_1,k_2\};\{k_3\}\big\} = \{s_4,s_3\}\\
S_{l_1}(p_5) &= \big\{\{k_1\};\{k_2\}\big\} = \{s_1,s_2\}\\
P_{l_1}\big(\{k_1,k_2,k_3\}\big) &= P_{l_1}(s_7) = \{p_1\} \\
P_{l_1}\big(\{k_1,k_2\}\big) &= P_{l_1}(s_4) = \{p_5\} 
\end{align*}

\section{IP Model}\label{Proposed_Model}

We propose an IP model $\mathcal{Z}_l$ that finds an optimal solution for a given laminar family $l\in\mathcal{L}_b$, where $b=|K|$. Since we are considering a fixed laminar family $l$, we are given $S(l)$, the set of subsets of $K$ that we use, and the partitions that are performed, as well as the {\it parent} and {\it child} sets of each set in $S(l)$. Therefore, the main purpose of our model is to choose the node where each partition is performed. Once we know where each partition is performed, it is easy to determine the arcs that minimize cost, since the problem reduces to finding a shortest path between the nodes where each subset $s\in S(l)$ begins and ends.

For notational simplicity, we do not index the variables of $\mathcal{Z}_l$ by $l$. Our proposed formulation $\mathcal{Z}_l$ is,

\begin{align}
&\mathcal{Z}_l :& \min&\sum_{a\in A}c_a\left(\sum_{s\in S(l)}f^{s}_a\right)\nonumber\\
&&\sum_{a\in\delta^+(i)}f^{s}_a-\sum_{a\in\delta^-(i)}f^{s}_a &= \hat{y}^{s}_i-\overline{y}^{s}_i &\forall i\in V,s\in S(l)\label{eq1}\\
&&w^{p}_i &=\overline{y}^{s}_i&\forall i\in V,s\in S(l): |s|\geq 2, p\in P_l(s)\label{eq2}\\
&&w^{p}_i &=\hat{y}^{s}_i&\forall i\in V,p\in P(l),s\in S_l(p)\label{eq3}\\
&&\sum_{i\in V}w^{p}_i &= 1 &\forall p\in P(l)\label{eq4}\\
&& \hat{y}_r^{K} &= 1 \label{eq5}\\
&& \hat{y}_i^{K} &= 0 &\forall i\in V\backslash\{r\}\label{eq6}\\
&& \overline{y}_{t_k}^{,k} &= 1 &\forall k\in K\label{eq7}\\
&& \overline{y}_i^{k}  &= 0&\forall k\in K,i\in V\backslash\{d_k\}\label{eq8}\\
&& f&\in\{0,1\}^{|A|\times|S(l)|}\label{eq9}\\
&& (\hat{y},\overline{y})&\in\{0,1\}^{|V|\times|S(l)|}\label{eq10}\\
&& w&\in\{0,1\}^{|V|\times|P(l)|}\label{eq11}
\end{align}

The variable $f_a^{s}$ is 1 if we send flow in arc $a$ for subset $s$, and 0 otherwise. This means, that all of the commodities that compose subset $s$ use arc $a$. The variable $\hat{y}_i^{s}$ is 1 if commodities in set $s$ {\it start} sharing a path in node $i$, and 0 otherwise. The variable $\overline{y}_i^{s}$ is 1 if commodities in set $s$ {\it end} sharing a path in node $i$, and 0 otherwise. Finally, $w_i^{p}$ is 1 if a partition $p$ is performed in node $i$, and 0 otherwise. Constraint (\ref{eq1}) is the flow conservation constraints for all the subsets that belong to laminar family $l$, which relates variables $f$, $\hat{y}$ and $\overline{y}$. Constraint (\ref{eq2}) states that if partition $p$ is performed in node $i$, then the subset $s$ that is partitioned has to stop sharing in node $i$. Constraint (\ref{eq3}) states that if partition $p$ occurs in node $i$, then the {\it child} sets of partition $p$ start to share in node $i$. Constraint (\ref{eq4}) imposes that the partitions can occur only in one node. Constraints (\ref{eq5}) and (\ref{eq6}) say that the set of all commodities start sharing in the root node. Finally, constraints (\ref{eq7}) and (\ref{eq8}) state that every commodity has to send its flow to its sink node.

The LP relaxation of $\mathcal{Z}_l$, which we refer to as $LP(\mathcal{Z}_l)$, is defined by the same objective function and set of linear equalities, but relaxing the integrality condition of the variables. This means that we replace $(\ref{eq9})$, $(\ref{eq10})$ and $(\ref{eq11})$ with

\begin{align}
f&\in[0,1]^{|A|\times|S(l)|}\label{eq12}\\
(\hat{y},\overline{y})&\in[0,1]^{|V|\times|S(l)|}\label{eq13}\\
w&\in[0,1]^{|V|\times|P(l)|}\label{eq14}
\end{align}

For all $l\in\mathcal{L}_b$, let $\mathcal{Q}_l^{IP}$ and $\mathcal{Q}_l^{LP}$ be the feasible regions of $\mathcal{Z}_l$ and $LP(\mathcal{Z}_l)$ respectively. For all $e\in E$, let $A(e)=\{a\in A:a\text{ is a directed arc of }e\in E\}$, and let $\chi\in\{0,1\}^{|E|}$. $\chi$ represents a solution to the Steiner tree problem in $G$, where $\chi_e=1$ if edge $e$ belongs to the solution.

\begin{definition}
Let $l\in\mathcal{L}_b$, and let $x=(f,w,\hat{y},\overline{y})$ be a feasible solution for $\mathcal{Q}_l^{IP}$. We define $\phi$ to be a mapping of $x$ to the original problem space, where
\begin{align*}
\phi(x) =\chi\in\{0,1\}^{|E|}
\end{align*}
given by
\begin{align*}
\chi_e = \text{min}\left\{1,\sum_{s\in S(l)}\sum_{a\in A(e)}f_a^s\right\}\ \ \ \ \text{for all }e\in E. 
\end{align*}
\end{definition}

\subsection{Structure of feasible solutions of $\mathcal{Q}_l^{IP}$ }\label{Sec:Structure_Q_IP}

We assume that $\mathcal{Q}_l^{IP}$ is feasible, and let $x=(f,w,\hat{y},\overline{y})$ be an arbitrary feasible solution for $\mathcal{Q}_l^{IP}$. A feasible solution can have two components, it must have an acyclic component, denoted by $x_{nc}$, and it may also have a cycle component, denoted by $x_c$. Then, $x =x_{nc}+x_c$. The acyclic component is always present and is a feasible solution to the problem, while the cycle component may not be present, i.e., we may have $x_c=0$. The cycle component is not feasible on its own, since it is only composed of cycles, and therefore it only fulfills the flow constraints (\ref{eq1}). We analyze each component separately.

First, consider the acyclic component, $x_{nc}=(f_{nc},w_{nc},\hat{y}_{nc},\overline{y}_{nc})$. Since this component has to be feasible, then by constraints (\ref{eq2})-(\ref{eq8}), $\hat{y}_{nc}^s$ and $\overline{y}_{nc}^s$ have exactly one non-zero component for all $s\in S(l)$,  which we call $i_s$ and $j_s$ respectively. There are two cases, either $i_s\ne j_s$, or $i_s=j_s$. In the first case, because of constraint (\ref{eq1}), there must be a path for $s$ from $i_s$ to $j_s$. In the second case, since $x_{nc}$ has no cycles, $f_{nc}^s=0$, or equivalently the set $s$ has no path. On the other hand, constraints (\ref{eq2})-(\ref{eq4}) ensure that the path of every parent set ends in the same node where the path of its child sets start. By constraints (\ref{eq5}) and (\ref{eq6}) we have that the set of all commodities must start in $r$, and by constraints (\ref{eq7}) and (\ref{eq8}), each singleton $k$ must end its path in $t_k$. Therefore, in every feasible solution, there is a path from $r$ to $t_k$ for all $k$ in $K$, which corresponds to the union of the paths of all sets containing $k$. Consequently, all Steiner trees with $(r,l)$ structure have a corresponding acyclic feasible solution to $\mathcal{Z}_l$, but not all acyclic solutions of $\mathcal{Z}_l$ maps to an $(r,l)$ structured Steiner tree. We will discuss the case in which a feasible acyclic solution is not an $(r,l)$ structured Steiner tree in Proposition \ref{Prop:Integer_Sol}.

Now, consider the cycle component. Let $x_{c}=(f_c,w_c,\hat{y}_c,\overline{y}_c)\ne0$, i.e, there is at least one set $\tilde{s}\in S(l)$ which has at least one cycle, and for simplicity of the argument, suppose that $\tilde{s}$ has only one cycle. The only way that this can happen is when the cycle uses arcs that are not in the path between $i_{\tilde{s}}$ and $j_{\tilde{s}}$, the path used by $\tilde{s}$ in the non-cycle component, otherwise it would be an infeasible solution because it would violate constraints (\ref{eq1}) and (\ref{eq9}). In the cycle component, $w_c$, $\hat{y}_c$ and $\tilde{y}_c$ are always 0, and if $x_c\ne 0$ then $f_c\ne 0$. Consequently, we have $w=w_{nc}$, $\hat{y}=\hat{y}_{nc}$, $\tilde{y}=\tilde{y}_{nc}$, and $f=f_{nc}+f_c$. If a solution does not contain cycles, then $f=f_{nc}$.

\begin{proposition}\label{Prop:Integer_Sol}
For all $l\in \mathcal{L}_b$, let $x$ be a feasible solution of $\mathcal{Q}_l$ that does not contain any cycles, then $\phi(x)$ is either an $(r,l)$ structured Steiner tree, or there exists a Steiner tree, with possible different $(r,l)$ structure, contained in the support of $\phi(x)$.
\end{proposition}

\begin{proof}
As was pointed out in Section \ref{Sec:Structure_Q_IP}, in the non-cycle component of every feasible solution $x$ of $\mathcal{Q}_l$, there is a path from $r$ to $t_k$ for every commodity $k\in K$. This is a property of every Steiner tree with an $(r,l)$ structure. But, there are also other cases where a feasible solution $x$ of $\mathcal{Q}_l$ does not map to an $(r,l)$ structured Steiner tree, as shown in Figure \ref{Fig:Comparisson_Feas_Sol}. We observe in Figure \ref{fig:Feas1}, a feasible solution $x$ of $\mathcal{Q}_{l_1}$ whose mapping $\phi(x)$ is not an $(r,l_1)$ structured Steiner tree\footnote{See Figure \ref{Fig:Splitting_Seq_k3} for description of laminar families $l_1$ and $l_3$.}. Note that commodities $k_2$ and $k_3$ share arcs, but not using set $\{k_2,k_3\}$ because this set is not in $l_1$.
\begin{figure}[H]
\begin{subfigure}{.5\textwidth}
  \centering
 \includegraphics[width=0.7\textwidth]{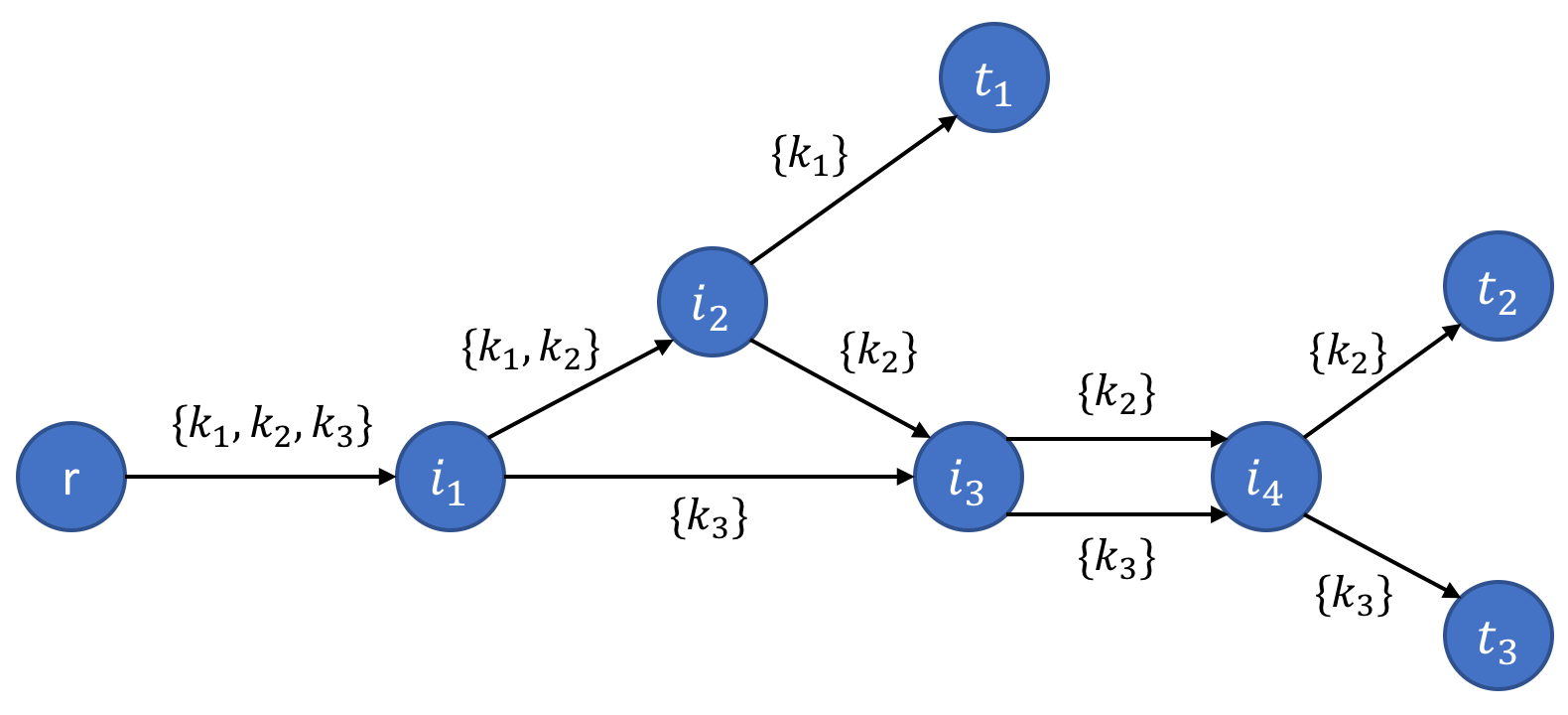}
  \caption{Feasible solution $x$ for $\mathcal{Q}_{l_1}$.}
  \label{fig:Feas1}
\end{subfigure}
\begin{subfigure}{.5\textwidth}
  \centering
 \includegraphics[width=0.7\textwidth]{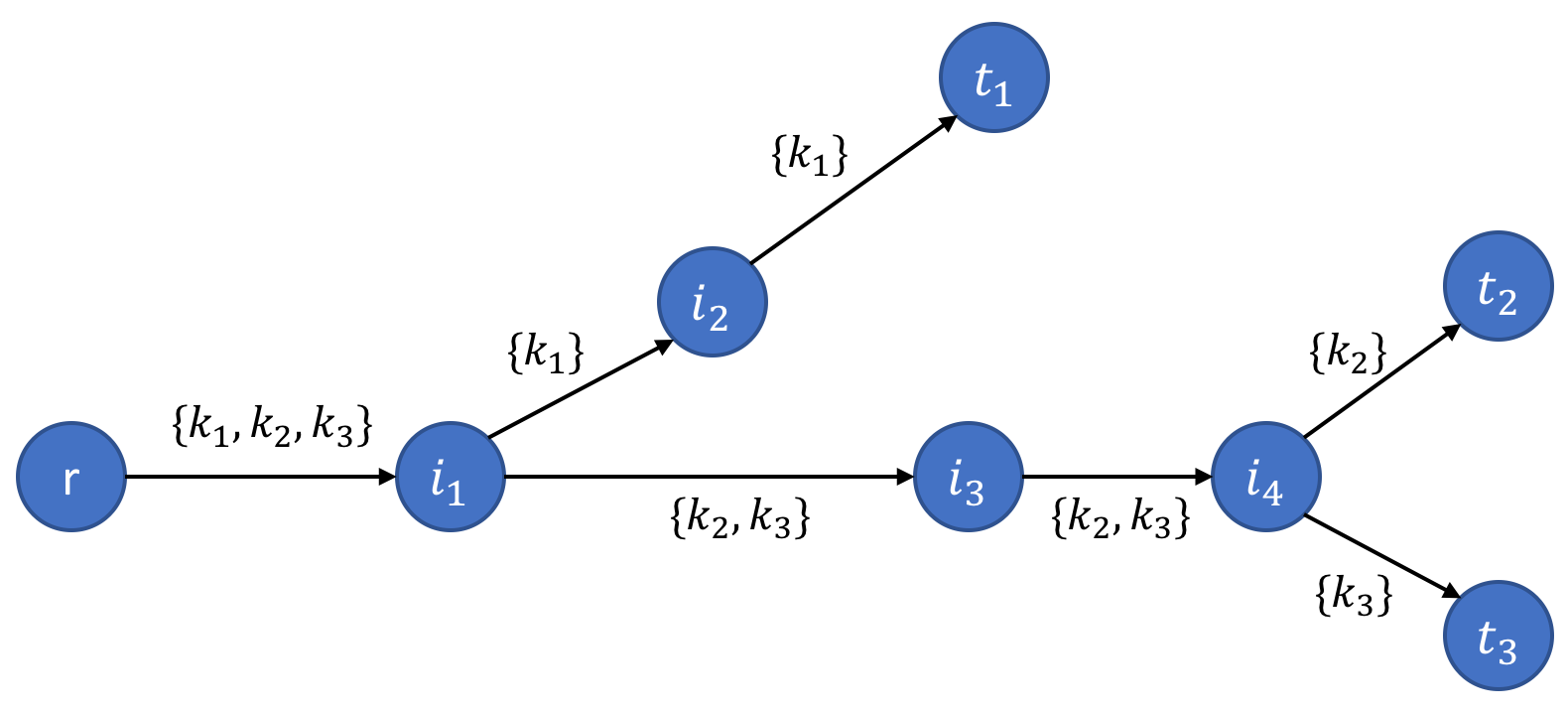}
  \caption{Steiner tree with $(r,l_3)$ structure, contained in support of $\phi(x)$.}
  \label{fig:Feas2}
\end{subfigure}
\caption{Feasible solution $x$ of $\mathcal{Q}_{l_1}$, whose mapping $\phi(x)$ is not an $(r,l_1)$ structured Steiner tree, and $(r,l_3)$ structured Steiner tree contained in support of $\phi(x)$.}
\label{Fig:Comparisson_Feas_Sol}
\end{figure}
Now, suppose $x\in\mathcal{Q}_l$ is acyclic and does not map to an $(r,l)$ structured Steiner tree. The mapping $\phi(x)$ tells us which edges of the original undirected graph $G$ are used by $x$. We can construct a subgraph defined by those edges, and we know there must be a Steiner tree in it, because in solution $x$ there is at least one path from $r$ to $t_k$ for all $k\in K$. Then, we can take any Steiner tree within the constructed subgraph, which will have a unique $(r,l^*)$ structure, with $l^*$ not necessarily equal to $l$. We can do this by taking any spanning tree in the constructed subgraph. For instance, Figure \ref{fig:Feas2} shows an $(r,l_3)$ structured Steiner tree within the support of $\phi(x)$ shown in Figure \ref{fig:Feas1}.\qed
\end{proof}

\subsection{Structure of feasible solutions of $\mathcal{Q}_l^{LP}$}\label{Sec:Structure_Q_LP}

The structure of the feasible solutions of $\mathcal{Q}_l^{LP}$, is similar to the solutions of $\mathcal{Q}_l^{IP}$, because they also have a non-cycle, and a cycle component. For simplicity, we use the same notation as in Section \ref{Sec:Structure_Q_IP}. We assume that $\mathcal{Q}_l^{LP}$ is feasible, and let $x=(f,w,\hat{y},\tilde{y})$ be an arbitrary solution of $\mathcal{Q}_l^{LP}$. Let $x=x_{nc}+x_c$, where $x_{nc}$ is the non-cycle component of $x$, and $x_c$ is the cycle component of $x$.

The analysis for this case is similar to the analysis in Section \ref{Sec:Structure_Q_IP}, but since the variables now can be fractional, we may have solutions where the non-cycle and the cycle component use the same arc $a$ for the same set $s$, but the sum of acyclic and cyclic components has to be at most 1 for all arcs and for all sets. Moreover, in the cycle component, we may have that two different cycles, for the same $s$, use the same arc $a$. Thus, the cycle part will be a collection of weighted cycles for each $s\in S(l)$, such that for every arc $a\in A$, the sum of acylic and cyclic flow in $a$ for $s$ is at most 1.

Let $LP(\mathcal{Z}_l)(\lambda)$ denote the formulation $LP(\mathcal{Z}_l)$, when we replace 1 by $\lambda\in (0,1]$ in constraints (\ref{eq4}), (\ref{eq5}), (\ref{eq7}), (\ref{eq12}), (\ref{eq13}), and  (\ref{eq14}). 

\begin{lemma}\label{Lemma:LP_Sol_Lambda}
For every $l\in \mathcal{L}_b$ and $\lambda\in(0,1]$, we have $LP(\mathcal{Z}_l)(\lambda)=\lambda \left[LP(\mathcal{Z}_l)(1)\right]$.
\end{lemma}

\begin{proof}
For any $\lambda\in(0,1]$, let $(f^\lambda,\hat{y}^\lambda,\overline{y}^\lambda,w^\lambda)$ denote an optimal solution for $LP(\mathcal{Z}_l)(\lambda)$. Then,
\begin{itemize}
\item $LP(\mathcal{Z}_l)(\lambda)\leq \lambda \left[LP(\mathcal{Z}_l)(1)\right]$. Since $(\lambda f^1,\lambda \hat{y}^1,\lambda \overline{y}^1,\lambda w^1)$ is a feasible solution to $LP(\mathcal{Z}_l)(\lambda)$, it holds that $LP(\mathcal{Z}_l)(\lambda)\leq \lambda \left[LP(\mathcal{Z}_l)(1)\right]$.
\item $LP(\mathcal{Z}_l)(\lambda)\geq \lambda \left[LP(\mathcal{Z}_l)(1)\right]$. Since $(\frac{1}{\lambda}f^\lambda,\frac{1}{\lambda}\hat{y}^\lambda,\frac{1}{\lambda}\overline{y}^\lambda,\frac{1}{\lambda}w^\lambda)$ is a feasible solution to $LP(\mathcal{Z}_l)(1)$, it holds that $\frac{1}{\lambda}\left[LP(\mathcal{Z}_l)(\lambda)\right]\geq LP(\mathcal{Z}_l)(1)$
\end{itemize}
Then, $LP(\mathcal{Z}_l)(\lambda)=\lambda \left[LP(\mathcal{Z}_l)(1)\right]$ for all $\lambda\in(0,1]$. \qed
\end{proof}

Our main result is stated in the following theorem.

\begin{theorem}\label{Thm:Proof_Integrality}
For all $l\in\mathcal{L}_b$, we have that $\text{Conv}(\mathcal{Q}_l^{IP})=\mathcal{Q}_l^{LP}$.
\end{theorem}

\begin{proof}
We will prove that for all cost vectors $c$, there exists an optimal solution to $LP(\mathcal{Z}_l)$ that is integral.

The following definitions will be used in the proof. For all $i\in V,s\in S(l)$, let $LP(\mathcal{Z}_l(i,s))$ be a subformulation of $LP(\mathcal{Z}_l)$, where the root is $i$ (instead of $r$), the terminal nodes are $t_k$ for all $k\in s$, and the splitting sequence is defined by the way $s$, and its subsets, are split in laminar family $l$. The cost vector of $LP(\mathcal{Z}_l(i,s))$ is the same as the one used in $LP(\mathcal{Z}_l)$, but only considering the components of $LP(\mathcal{Z}_l)$ that are present in $LP(\mathcal{Z}_l(i,s))$.  Let $x^*(i,s)$ be an optimal solution to $LP(\mathcal{Z}_l(i,s))$.

Let $x^*=(f^*,\hat{y}^*,\overline{y}^*,w^*)$ be an optimal solution of $LP(\mathcal{Z}_l)$. Note that if $w^*$ is integer, then $\hat{y}^*$ and $\overline{y}^*$ are integer because of constraints (\ref{eq4})-(\ref{eq8}). This implies, that for all $s\in S(l)$, there is only one non-zero component in vectors $\hat{y}^s$ and $\overline{y}^s$, which is equal to 1. Let $i_s$ and $j_s$ be those indexes, respectively. Now, for all $s\in S(l)$, constraints (\ref{eq1}) and (\ref{eq9}) correspond to the shortest $i_s$-$j_s$ path polytope, which has only integer extreme points. Since each of these polytopes is independent of each other, we have that $f^*$ is integer. Therefore, if $w^*$ is integer, then $x^*$ is integer.

Now, suppose that $w^*$ is fractional. This implies, that for at least one set $s\in S(l)$, $\overline{y}^s$ is fractional. Note that, we may have that $\hat{y}^s$ is also fractional, but for at least one $s\in S(l)$, we claim that $\hat{y}^s$ is integer, and $\overline{y}^s$ is fractional. The justifications for this claim is the following. Suppose the partitions in $P(l)$ are ordered in decreasing order by the cardinality of the set they split, i.e., the first element of $P(l)$ splits $K$, and the last set of $P(l)$ splits a set of cardinality 2. Now, consider $w^*$, and let $\dot{p}$ be the first element in $P(l)$ such that $w^{\dot{p}}$ is fractional. Let $\dot{s}$ be the set that $\dot{p}$ splits. Then, it is clear that $\hat{y}^{\dot{s}}$ is integer and $\overline{y}^{\dot{s}}$ is fractional. 

Let $\hat{s}$ be the largest set in $S(l)$ such that $\hat{y}^{\hat{s}}$ is integral, and $\overline{y}^{\hat{s}}$ is fractional, and for simplicity of the argument, suppose $\hat{s}=K$. Let $I(\hat{s})$ be the set of indexes, such that $\overline{y}^{\hat{s}}_i=\lambda_i>0$ for all $i\in I(\hat{s})$, and $\sum_{i\in I(\hat{s})}\lambda_i=1$. Let $\hat{p}\in P(l)$, be the partition that splits $\hat{s}$ in laminar family $l$. By constraint (\ref{eq2}) we have that $w_i^{\hat{p}}=\lambda_i$ for all $i\in I(\hat{s})$, as shown in Figure \ref{Thm1:Fig_1}.

\begin{figure}[H]
\begin{center}
\includegraphics*[width=0.5\textwidth]{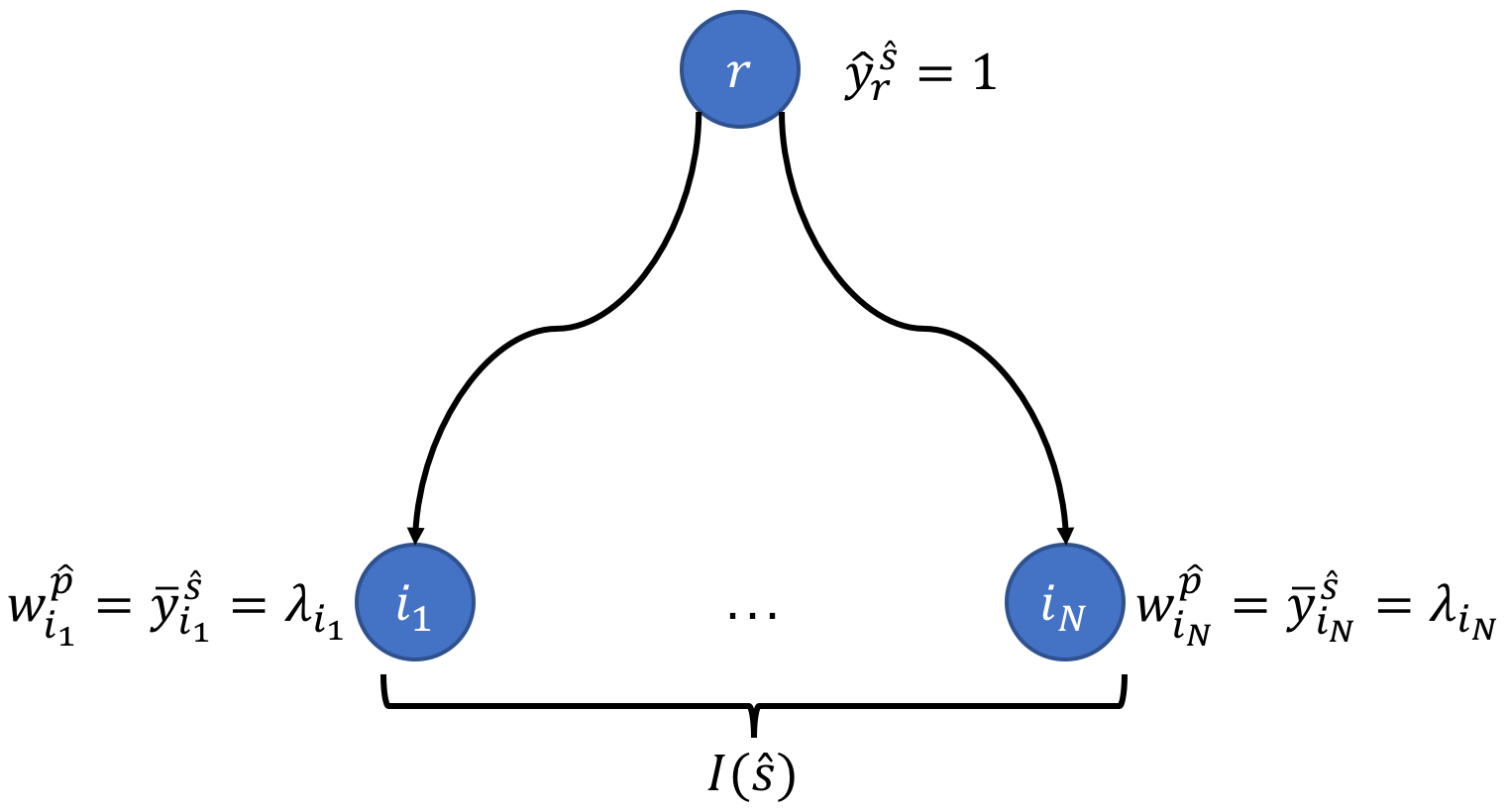}
\caption{Solution illustration.}
\label{Thm1:Fig_1}
\end{center}
\end{figure}

For all $i\in I(\hat{s})$ and $s\in S_l(\hat{p})$\footnote{Recall that $S_l(p)$ is the set of all child subsets of $p$ in laminar family $l$.}, by constraint (\ref{eq3}), we have that $\hat{y}^s_i=\lambda_i$. Now, take a fixed $s\in S_l(\hat{p})$, and a fixed $i\in I(\hat{s})$. Because $\hat{y}^s_i=\lambda_i$, then we know that within solution $x^*$, we are sending $\lambda_i$ units of flow from node $i$ to $\{t_k\}_{k\in s}$. Since $x^*$ is optimal, then by Lemma \ref{Lemma:LP_Sol_Lambda}, that part of the solution has to correspond to $\lambda_i x^*(i,s)$, otherwise, we can improve solution $x^*$. On the other hand, within solution $x^*$ we are sending $\lambda_i$ units of flow from $r$ to $i$, for all $i\in I(\hat{s})$. Let $x^*(r,i,\hat{s})$ be the lowest cost solution sending 1 unit of flow from $r$ to $i\in I(\hat{s})$ for set $\hat{s}$, then, by Lemma \ref{Lemma:LP_Sol_Lambda}, $\lambda_i x^*(r,i,\hat{s})$ corresponds to the lowest cost solution sending $\lambda_i$ units of flow from $r$ to $i\in I(\hat{s})$ for set $\hat{s}$, otherwise we can improve solution $x^*$. Figure \ref{Thm1:Fig_2} shows a representation of the last statement.
\begin{figure}[H]
\begin{center}
\includegraphics*[width=0.6\textwidth]{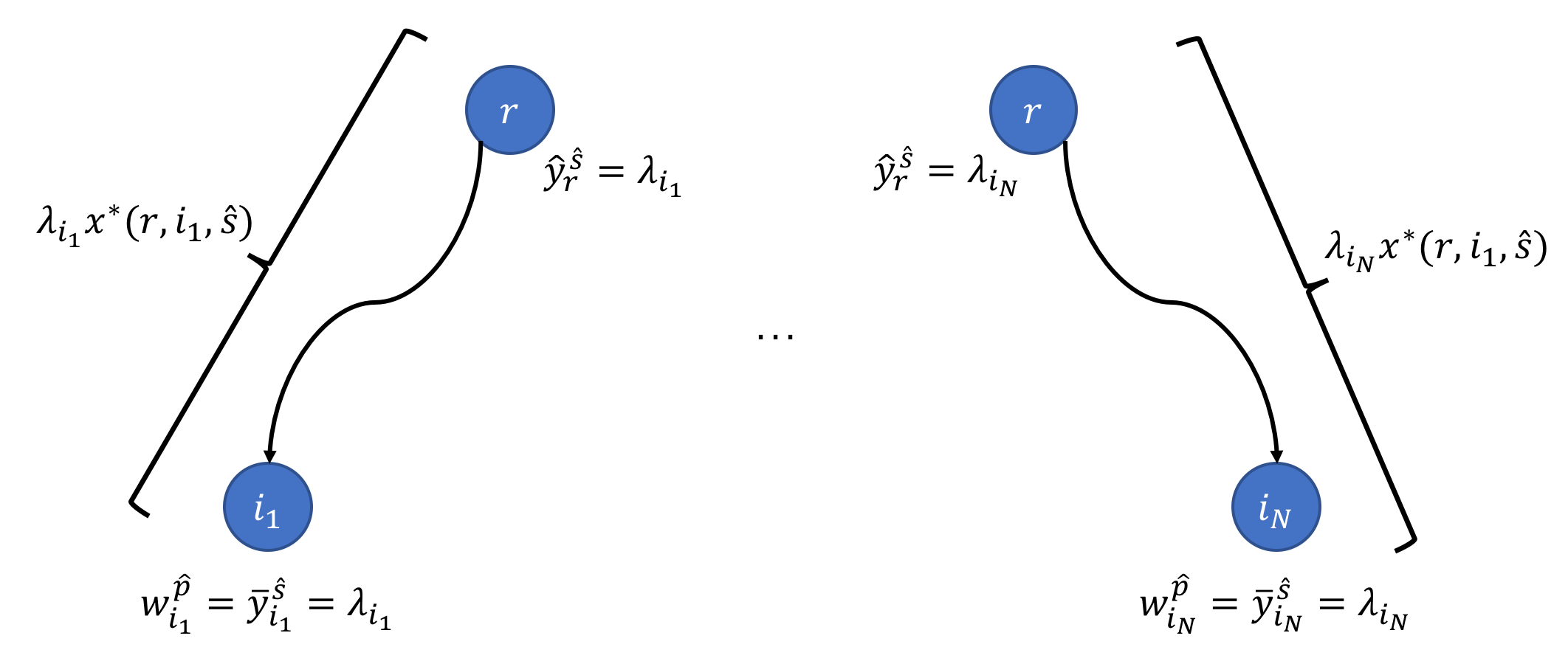}
\caption{Solution illustration.}
\label{Thm1:Fig_2}
\end{center}
\end{figure}
Putting everything together, we have
\begin{align*}
x^* &= \sum_{i\in I(\hat{s})}\lambda_i\left[x^*(r,i,\hat{s})+\sum_{s\in S_l(\hat{p})}x^*(i,s)\right]
\end{align*}
But, for all $i\in I(\hat{s})$, $x^*(r,i,\hat{s})+\sum_{s\in S_l(\hat{p})}x^*(i,s)$ is a feasible solution. Then, $x^*$ is a convex combination of $\{x^*(r,i,\hat{s})+\sum_{s\in S_l(\hat{p})}x^*(i,s)\}_{i\in I(\hat{s})}$, and therefore, it is not an extreme point. Note that, $x^*(r,i,\hat{s})$ and $\sum_{s\in S_l(\hat{p})}x^*(i,s)$ may contain cycles. Furthermore, if $\hat{s}\ne K$, we can use the same argument, but we have to consider the solution of all the sets that contain $\hat{s}$ as a subset. By definition of $\hat{s}$, in the solution of all those sets, $\hat{y}$ and $\overline{y}$ will be integers, and therefore, the entire solution of all those sets has to be integer. Thus, we can include that part of the solution in $x^*(r,i,\hat{s})+\sum_{s\in S_l(\hat{p})}x^*(i,s)$ for all $i\in I(\hat{s})$, and use the same argument as before. Consequently, all extreme point solutions have $\hat{y},\overline{y}$, and $w$ integral, and therefore, all extreme points of $\mathcal{Q}_l^{LP}$ are integral. \qed
\end{proof}

As a corollary, we can solve the entire Steiner tree problem by optimizing each laminar family subproblem independently, and taking the solution with the lowest cost. For simplicity of notation, we define $\mathcal{Q}_l=\mathcal{Q}_l^{IP}$, which are the extreme points of $\mathcal{Q}_l^{LP}$.

\begin{corollary}\label{Corollary:Sol_Entire_Problem}
Let $c\geq 0$ be the cost vector of a Steiner tree problem instance. Let $x^l=\argmin\{c^Tx:x\in\mathcal{Q}_l\}$, and $v^l=c^Tx^l$ for all $l\in\mathcal{L}_b$. Then, $\phi(x^{l^*})$ is a minimum cost Steiner tree, where $l^*=\argmin\{v^l:l\in\mathcal{L}_b\}$.
\end{corollary}

\begin{proof}
Note that, since the cost is non-negative, then the optimal solution for each subproblem is acyclic. By Proposition \ref{Prop:Integer_Sol}, we know that $\phi(x^{l^*})$ is either an $(r,l^*)$ structured Steiner tree, or $\phi(x^{l^*})$ contains a Steiner tree, with a different structure. But, since the cost vector is non-negative, then $\phi(x^{l^*})$ has to be a Steiner tree, otherwise, we can take any Steiner tree contained in the support of $\phi(x^{l^*})$, which will have a lower cost since it uses a subset of the edges used by $\phi(x^{l^*})$. Therefore, $\phi(x^{l^*})$ is a minimum cost Steiner tree. \qed
\end{proof}

Note that using a standard disjunctive programming argument, we can write an extended formulation for the entire problem based on the formulations for each laminar family \cite{balas1998disjunctive,Conforti2014IntProg}.

\subsection{Problem size}

The downside of our proposed formulation is the number of laminar families, which grows very fast as the number of terminals increase. In this section we provide an upper bound on the number of laminar families that we have to consider, and we show that this upper bound is tight.

In the following proposition, we use the notion of a {\it full binary} tree. Note that, a tree $T$ is called a {\it full binary} tree if every node in $T$ has either zero or two children.

\begin{proposition}\label{Prop:Number_Lam_Families}
It suffices only to consider laminar families in $\mathcal{L}_b$ whose tree representation corresponds to a full binary tree.
\end{proposition}

\begin{proof}
Let $l_1$ be a laminar family in $\mathcal{L}_b$ and let $T(l_1)$ be the tree representation of $l_1$, such that $T(l_1)$ is not a full binary tree. Therefore, there must be a node in $T(l_1)$ that has at least 3 child nodes, which implies that there must be a set in $l_1$ that is partitioned into at least 3 sets, let $\hat{s}$ be such set. Suppose that the {\it child} sets of $\hat{s}$ in $l_1$ are $s_1,s_2,\ldots,s_t$ with $t\geq 3$. For simplicity, assume $t=3$. Now, let $l_2$ be the laminar family in $\mathcal{L}_b$ such that $S(l_2)=S(l_1)\cup{s'}$ where $s'=s_1\cup s_2$. This means that all partitions in $l_1$ and $l_2$ are the same, but in $l_1$, $\hat{s}$ is partitioned into $s_1,s_2,s_3$, and in $l_2$, $\hat{s}$ is partitioned into $s'$ and $s_3$, and $s'$ is partitioned into $s_1$ and $s_2$. Now, let $x_1=(f_1,w_1,\hat{y}_1,\overline{y}_1)$ be an optimal solution to $\mathcal{Z}_{l_1}$. We claim that we can construct a feasible solution to $\mathcal{Z}_{l_2}$ using $x_1$. Let $i^*$ be the node where $\hat{s}$ is split in $x_1$, then we define $x_2=(f_2,w_2,\hat{y}_2,\overline{y}_2)\in \mathcal{Q}_{l_2}$ as follows,
\begin{align*}
(f_2)_a^s &= \left\{\begin{array}{ll}
(f_1)_a^s & \forall a \in A, s\in S(l_1)\\
0 & \forall a\in A, s=s'
\end{array}
\right.\\
(\hat{y}_2)_i^s &= \left\{\begin{array}{ll}
(\hat{y}_1)_i^s& \forall i \in V, s\in S(l_1)\\
1 &  i=i^*, s=s'\\
0 & \forall i\in V\backslash\{i^*\}, s=s'
\end{array}
\right.\\
(\overline{y}_2)_i^s &= \left\{\begin{array}{ll}
(\overline{y}_1)_i^s& \forall i \in V, s\in S(l_1)\\
1 &  i=i^*, s=s'\\
0 & \forall i\in V\backslash\{i^*\}, s=s'
\end{array}
\right.
\end{align*}
Note that the values of the $w_2$ vector will be determined by the values of $\hat{y}_2$ and $\overline{y}_2$. Moreover, $x_2$ is a feasible solution to $\mathcal{Z}_{l_2}$, whose cost is the same as the cost of $x_1$. Hence, $\mathcal{Z}_{l_2}\leq \mathcal{Z}_{l_1}$. The same argument holds for $t>3$ as well. Consequently, we only need laminar families where each non-singleton set is partitioned into two proper subsets, which correspond to laminar families whose tree representation is a full binary tree. 
\qed
\end{proof}

Using Proposition \ref{Prop:Number_Lam_Families}, we can reduce the number of laminar families by considering only laminar families whose tree representations are full binary trees. For simplicity of notation, we denote by $\mathcal{L}_b$ the reduced set of laminar families, when we have $b+1$ terminals. We are interested in expressing $|\mathcal{L}_b|$ as a function of $b$. Note that, $|\mathcal{L}_b|$ is equivalent to the number of full binary labeled trees with $b$ leaves. It is known that (see \cite{billera2001geometry}, and  \cite{stanley1999enum} Chapter 5.2.6).
\begin{align*}
|\mathcal{L}_b| = (2b-3)!! = \frac{(2(b-1))!}{2^{b-1}(b-1)!}
\end{align*}
Note, however, that $|\mathcal{L}_b|$ is not always achieved, since there are graphs that do not yield some laminar families of $\mathcal{L}_b$, which is why $\frac{(2(b-1))!}{2^{b-1}(b-1)!}$ is an upper bound on the number of laminar families to consider. Now, recall Stirling's formula, which gives lower and upper bounds for $n!$. For any $n\in\Z^+$ we have
\begin{align*}
\sqrt{2\pi}n^{n+\frac{1}{2}}e^{-n}\leq n!\leq en^{n+\frac{1}{2}}e^{-n}
\end{align*}
Then, 
\begin{align*}
|\mathcal{L}_b| &= \frac{(2(b-1))!}{2^{b-1}(b-1)!}\\
&\leq \frac{e(2(b-1))^{2(b-1)+\frac{1}{2}}e^{-2(b-1)}}{2^{b-1}\sqrt{2\pi}(b-1)^{(b-1)+\frac{1}{2}}e^{-(b-1)}}\\
&= \frac{e2^{(b-1)}(b-1)^{(b-1)}e^{-(b-1)}}{\sqrt{\pi}}\\
&= \frac{e}{\sqrt{\pi}}\left(\frac{2(b-1)}{e}\right)^{(b-1)}
\end{align*}
We conclude that $|\mathcal{L}_b|$ grows as $\mathcal{O}\left(\left(\frac{2b}{e}\right)^{b}\right)$. Although this is a big number, it is small compared to the total number of laminar families. In fact, it can be shown that as $b$ increases, the ratio between the number of laminar families corresponding to full binary trees and the total number of laminar families, goes to 0.

Now, we analyze the size of each subproblem $LP(\mathcal{Z}_l)$ for $l\in\mathcal{L}_b$. By Proposition \ref{Prop:Number_Lam_Families}, we use laminar families whose tree representation $T(l)$ is a full binary tree. An important property of full binary trees is that, the number of internal nodes is $L-1$, where $L$ is the number of leaves. In our case, there are $b$ leaves, which implies $b-1$ internal nodes, and $2b-1$ nodes in the entire tree. This implies that $|S(l)|=2b-1$ for all $l\in\mathcal{L}_b$. Moreover, the number of non-singleton sets corresponds to the number of internal nodes in $T(l)$ which is $b-1$, and also is the number of partitions in $l$. Finally, since we only consider $l$ such that $T(l)$ is a full binary tree, the number of {\it child} sets of each partition is exactly 2. Using these facts, the number of variables and constraints are the same for all $l\in\mathcal{L}_b$, and the number of variables is  $\mathcal{O}(nb+mb)$ and the number of constraints is $\mathcal{O}(nb+b)$, where $n=|V|$, $m=|A|$ and $b+1=|R|$. Therefore, $LP(\mathcal{Z}_l)$ is of polynomial size in the input data. Consequently, we have that each subproblem is polynomially solvable, but the number of subproblems grow as $\mathcal{O}\left(\left(\frac{2b}{e}\right)^{b}\right)$, which implies that the problem is fixed-parameter tractable with respect to the number of terminal nodes. This last remark is consistent with previous results \cite{dreyfus1971steiner,feldmann2017complexity}.

\section{Computational Results}\label{Results}

We provide computational results to evaluate the performance of our proposed formulation. We use instances from the SteinLib library \cite{koch2001steinlib}. Table \ref{Table_Running_Times} shows the number of laminar families as a function of the number of terminal nodes, and it shows how long it will take to solve the entire problem if each subproblem takes one second to execute, and they are solved sequentially.
\begin{table}[H]
\begin{center}
\begin{tabular}{c r r}
\hline
	Number of terminals & $|\mathcal{L}_b|$ & Running time \\  \hline
	3 & 1 & 1 seconds \\
	4 & 3 &  3 seconds \\
	5 & 15 & 15 seconds \\
	6 & 105 & 1.75 minutes \\
	7 & 945 & 15.75 minutes \\
	8 & 10,395 & 2.89 hours \\
	9 & 135,135 & 37.54 hours \\
	10 & 2,027,025 & 23.46 days \\
	11 & 34,459,425 & 1.09 years \\
	 \hline
\end{tabular}
\end{center}
\caption{Number of terminals, maximum number of laminar families, and total execution time if each subproblem takes 1 second to solve.}
\label{Table_Running_Times}
\end{table}
As we can see from Table \ref{Table_Running_Times}, the expected time to run instances with 9 or more terminals will be very large, even if the execution time for each subproblem is 1 second. Thus, we have only run experiments on instances with at most 8 terminal nodes.

Our code is implemented in {\it Python} using {\it Gurobi} 7.5.1 as a solver. All tests were performed on a laptop with an {\it Intel Core i7} (3.3 GHz) processor, and 16 GB of memory, using the {\it macOS Sierra} operating system. As we previously discussed, we can solve each laminar family subproblem independently, and therefore, they can be solved in parallel, but we implemented our code such that we run each subproblem sequentially.

Table \ref{Table_Results} summarizes the results. The first two columns correspond to the test set and instance name, the third column shows the number of nodes, the fourth column shows the number of arcs, the fifth and the sixth columns present the number of terminal nodes and total number of laminar families. Finally, the last two columns present the average execution time for each subproblem, and the total execution time for all the models, respectively. All the reported times are in seconds.
\begin{table}[H]
\begin{center}
\begin{tabular}{c r r r c r c r}
\hline
	Test set & Instance & $|V|$ & $|A|$ & $|R|$ & $|\mathcal{L}_b|$ & Subproblem time & Entire problem time \\  \hline
	LIN & {\it lin01} & 53 & 160 & 4 & 3 & 0.003 s & 0.01 s\\
	LIN & {\it lin02} & 55 & 164 & 6 & 105 & 0.005 s & 0.48 s\\
	LIN & {\it lin04} & 157 & 532 & 6 & 105 & 0.021 s & 2.18 s\\
	LIN & {\it lin07} & 307 & 1,052 & 6 & 105 & 0.060 s & 6.26 s\\
	I160 & {\it i160-033} & 160  & 640 & 7 & 945 & 0.016 s & 15.43 s\\
	I160 & {\it i160-043} & 160 & 5,088 & 7 & 945 & 0.151 s & 142.63 s\\
	I160 & {\it i160-045} & 160 & 5,088 & 7 & 945 & 0.164 s & 154.74 s\\
	LIN & {\it lin03} & 57 & 168 & 8 & 10,395 & 0.007 s & 71.26 s\\
	PUC & {\it cc3-4p} & 64 & 576 & 8 & 10,395 & 0.018 s & 190.89 s\\
	PUC & {\it cc3-4u} & 64 & 576 & 8 & 10,395 & 0.019 s & 192.65 s\\
	I320 & {\it i320-011} & 320 & 3,690 & 8 & 10,395 & 0.197 s & 1,806.00 s \\
	I320 & {\it i320-043} & 320 & 20,416 & 8 & 10,395 & 0.910 s & 9,468.00 s \\
	\hline
\end{tabular}
\end{center}
\caption{Instance details and execution times in seconds.}
\label{Table_Results}
\end{table}

\section{Conclusions and Future Work}\label{Conclusions}

We propose an LP based approach to the Steiner tree problem. The proposed approach consists of solving a set of independent IPs. The best solution among the set of IPs corresponds to an optimal Steiner tree. Each IP is polynomial in the size of the underlying graph, and we prove that the LP relaxation of each IP is integral, so each IP can be solved as a linear program. The main issue is that the number of IPs to solve grows as $\mathcal{O}\left(\left(\frac{2b}{e}\right)^{b}\right)$ where $b+1$ is the number of terminal nodes. Consequently, we are able to solve the Steiner tree problem by solving a polynomial number of LPs, when the number of terminals is fixed. This is consistent with previous results  \cite{dreyfus1971steiner,feldmann2017complexity}.

Future research might be directed in trying to develop tools to deal with the large number of laminar families that need to be considered. One approach could be to use Column Generation to solve this problem. We can start with a small set of the laminar families, and then try to add new laminar families to the problem in a smart way. The main difficulty with this approach, is to develop an efficient pricing problem to determine which laminar families must be added to the problem.

Another line of research is to develop new approximation algorithms using our formulation. Since solving our formulation for a small set of laminar families is fast, then one may try to find laminar families which are good candidates to be optimal. For instance, one can look at the linear relaxation of the {\it Bidirected Cut} formulation, which is easy to solve, and then identify all the laminar families used by the fractional solution obtained. We can use those laminar families as candidates for an optimal solution. 

In addition, we can use our formulation to improve solutions delivered by any approximation algorithm. This can be done, because every solution delivered by an approximation algorithm is a Steiner tree. We can identify the $(r,l)$ structure of such a tree, and then find a minimum cost Steiner tree with $(r,l)$ structure by solving $LP(\mathcal{Z}_l)$, which will take polynomial time since we are only solving one subproblem.

Another research path is to develop an algorithm to solve each subproblem efficiently. The structure of the optimal solution for each problem is well characterized. For each subproblem, the problem reduces to finding the splitting nodes, and then taking the union of shortest paths. Furthermore, we can use this algorithm to construct a local search algorithm. Suppose we have a solution for a given laminar family $l\in\mathcal{L}_b$. There are laminar families in $\mathcal{L}_b$ that have a similar topology to that of $l$. Then, the idea is to develop an algorithm that searches for an improving laminar family in the neighborhood of $l$.

Finally, a different line of  research is to study our proposed approach in special graphs. There are important results in the literature for special graphs, in particular there are results that provide an upper bound on the integrality gap for well-known formulations. In these type of graphs, we may find a bound on how bad the optimal solution of the subproblem is, for the ``worst'' structure compared with the optimal solution for the entire problem. Moreover, there are results that show that for special cases of planar graphs, there is an algorithm to solve the problem in polynomial time \cite{bern1991polynomially}. These results, may imply that for those type of graphs, the number of tree structures to consider is limited.

\section*{Acknowledgement}
This research was partially supported by Office on Naval Research grants N00014-15-1-2078 and N00014-18-1-2075 to the Georgia Institute of Technology, and by the CONICYT (Chilean National Commission for Scientific and Technological Research) through the Doctoral Fellowship program ``Becas Chile'', Grant No. 72160393
\newpage
\bibliography{References}

\end{document}